\documentclass{amsart}
\usepackage[utf8]{inputenc}
\usepackage[margin=1in]{geometry}
\usepackage{amsthm, amsmath, amssymb, bm, xcolor, comment, mathtools, indentfirst, hyperref,soul,float,graphicx}
\usepackage{tikz}

\theoremstyle{plain}
\newtheorem{thm}{Theorem}[section]

\newtheorem{prop}[thm]{Proposition}
\newtheorem{lemma}[thm]{Lemma}

\newtheorem{cor}{Corollary}[thm]

\newtheorem{manualtheoreminner}{Theorem}
\newenvironment{manualtheorem}[1]{%
  \renewcommand\themanualtheoreminner{#1}%
  \manualtheoreminner
}{\endmanualtheoreminner}

\theoremstyle{remark}
\newtheorem{rmk}[thm]{Remark}

\theoremstyle{definition}
\newtheorem{ex}[thm]{Example}
\newtheorem{smallex}[thm]{Example}
\newtheorem{defn}[thm]{Definition}
\newtheorem{ques}[thm]{Question}

\newtheorem{setting}[thm]{Setting}

\newcommand{\ZZ}{\mathbb{Z}}
\newcommand{\CC}{\mathbb{C}}

\newcommand{\kk}{\mathsf{k}}

\newcommand{\supp}{\operatorname{supp}}

\newcommand{\diffpower}[2]{{#1}^{\langle #2 \rangle}}
\newcommand{\diffpowerA}[2]{{#1}^{\langle #2 \rangle_A}}
\newcommand{\diffideal}[1]{\overline{#1}^{\text{diff}}}
\newcommand{\diffidealA}[1]{\overline{#1}^{\text{diff}_A}}

\newcommand{\invrs}{^{-1}}
\newcommand{\zdgeq}{\mathbb{Z}^d_{\geq 0}}

\definecolor{ao(english)}{rgb}{0.0, 0.5, 0.0}

\DeclareMathOperator{\End}{End}

\title{Asymptotic Behavior of Differential Powers}
\author{Jennifer Kenkel, Lillian McPherson, Janet Page, \\ Daniel Smolkin, Monroe Stephenson, and Fuxiang Yang}
\date{August 2021}
\begin{document}

\maketitle

\begin{abstract}
   In this paper, we study the \emph{differential power} operation on ideals, recently defined in \cite{MR3779569}. We begin with a focus on monomial ideals in characteristic 0 and find a class of ideals whose differential powers are eventually principal. We also study the containment problem between ordinary and differential powers of ideals, in analogy to earlier work comparing ordinary and symbolic powers of ideals. 
   
   We further define a possible closure operation on ideals, called the \emph{differential closure}, in analogy with integral closure and tight closure. We show that this closure operation agrees with taking the radical of an ideal if and only if the ambient ring is a simple $D$-module. 
\end{abstract}
\tableofcontents

\section{Introduction}
Given a commutative ring $R$ and an ideal $I \subseteq R$, there are various notions of powers of the ideal $I$. Two well-studied such notions are the $n^{th}$ \emph{ordinary power} $I^n$, which is the ideal generated by the products of $n$ elements of $I$, and, in characteristic $p>0$, the $e^{th}$ \emph{Frobenius power}, $I^{[p^e]}$, which is the ideal generated by $({p^e})^{th}$ powers of elements in $I$ \cite{HochsterHunekeTight}.  

One less-understood notion of ideal powers is the \emph{differential power} of an ideal. The $n$th differential power of an ideal $I$, with respect to a base ring $A$, is  denoted $\diffpowerA{I}{n}$ and defined using the ring of $A$-linear differential operators on $R$. Differential powers of ideals have recently seen a renewed interest, having been used to been used to generalize the celebrated Zariski-Nagata Theorem (\cite{MR3779569}, \cite{ZariskiNagataMixedChar}). For prime ideals in polynomial rings of characteristic zero, the differential power of an ideal coincides with the symbolic power of an ideal \cite{MR3779569}; one can even use a variant of the differential power to extend this result to any algebra essentially of finite type over a perfect field \cite[Theorem 4.6]{NoetherianOps}. Differential powers have also been called \emph{$n^{\textrm{th}}$ primitive ideals} in the literature \cite{HigherRelativePrimitiveIdeals}.

In this paper, we study the differential powers $\diffpowerA{I}{n}$ of an ideal $I$ by examining their behavior as $n$ grows arbitrarily large. We approach this problem from three directions. One basic question is whether the number of generators of $\diffpowerA{I}{n}$ must grow as $n$ grows. We show that, on the contrary, one can find ideals with arbitrarily many generators whose differential powers are eventually principle, even in the polynomial ring setting:

\begin{manualtheorem}{\ref{principal2d}, \ref{principal3d}}
Let $R$ be a polynomial ring in 2 or 3 variables over a field of characteristic $0$. If $I$ is an ideal such that $\sqrt{I} = (x_1x_2)$ over $\kk[x_1,x_2]$ or $\sqrt{I} = (x_1x_2x_3)$ over $\kk[x_1,x_2,x_3]$, then there exists a positive integer $N$ such that $\diffpower{I}{N}$ is principal.
\end{manualtheorem}

Next, we continue to work in the polynomial ring setting and study a containment problem on differential versus ordinary powers of ideals. Commutative algebraists have long studied which symbolic powers of an ideal are contained in which ordinary powers. For instance, a celebrated result of Ein, Lazarsfeld, and Smith shows that the symbolic power $\mathfrak p^{(dn)}$ is contained in $\mathfrak p^n$ whenever $\mathfrak p$ is a prime ideal in a regular $\mathbb C$-algebra of dimension $d$; see \cite{MR1826369}. Inspired by these results, we ask which differential powers of a monomial ideal are contained in which ordinary powers. For instance, we have the following:

\begin{manualtheorem}{\ref{irreduciblecontainment}} Let $R = \kk[x_1, \ldots, x_d]$ where $\kk$ is a field of characteristic 0. Let $I = (x_{i_1}^{\alpha_1},x_{i_2}^{\alpha_2},\ldots, x_{i_m}^{\alpha_m})$. If $c = \max\{\alpha_1,\ldots, \alpha_m, m+1\}$, then $\diffpower{I}{cn} \subseteq I^n$ for all $n$.
\end{manualtheorem}

On the other hand, we also show that a \emph{uniform} containment result analogous to that of \cite{MR1826369} is too much to hope for:

\begin{manualtheorem}{\ref{prop:noContainment}}
Let $R = \kk[x_1, \ldots, x_d]$ where $\kk$ is a field of characteristic 0. There does not exist a polynomial $p(n)$ such that $\diffpower{I}{p(n)}\subseteq I^n$ for all ideals $I \subseteq R$ and all $n\in \ZZ_{\ge 0}$.
\end{manualtheorem}

In section \ref{sec:diffClosure}, we go in a different direction and propose a notion of ideal closure using differential powers. Other studied powers of ideals, the ordinary power and the Frobenius power, lead to natural closure operations, the \emph{integral closure} \cite{HunekeSwansonIntegral}
and the \emph{tight closure} \cite{HochsterHunekeTight} respectively. Inspired by the definitions of these closure operations, we use differential powers to define a new closure operation, the \emph{differential closure}. The differential closure of an ideal $I$ in an $A$-algebra $R$ is denoted $\diffidealA{I}$. We show that this operation detects the simplicity or non-simplicity of $R$ as a module over its ring of $A$-linear differential operators, $D_A(R)$:

\begin{manualtheorem}{\ref{radicalDmodule}}
Let $R$ be a domain essentially of finite type over a Noetherian ring $A$. Then $\sqrt{I} = \diffidealA{I}$ for all ideals $I \subseteq R$ if and only if $R$ is a simple $D_A(R)-$module.
\end{manualtheorem}

The simplicity or non-simplicity of $R$ as a $D_A(R)$ module is closely tied to the singularities of $R$. For instance, it is shown in \cite{smithdmod} that, if $R$ is an $F$-split domain essentially of finite type over a perfect field $\kk$, then $R$ is strongly $F$-regular if and only if $R$ is simple as a $D_\kk(R)$-module.

{\bf Acknowledgements.}
This project started as a Research Experience for Undergraduates (REU) at the University of Michigan mentored by the first, third, and fourth named authors. We would like to acknowledge the University of Michigan Department of Mathematics, and in particular David Speyer, for organizing the REU program and arranging funding. We also thank Jack Jeffries and Elo\'isa Grifo for helpful discussions. The research was supported by the UM REU executive committee funds, Karen Smith's NSF grant DMS-1801697, and David Speyer's NSF grant DMS-1855135.

\section{Background on Differential Operators and Differential Powers}
In order to study differential powers and differential closure, we must familiarize ourselves with definitions of differential operators on a commutative algebra. We refer the reader to \cite{MR1356713} and \cite[Lecture 17]{srikanth_2007} for more details and proofs. Throughout this section, let $R$ be a commutative $A$-algebra. 

\begin{defn}\label{defcommutator}
The commutator of two elements $P, Q \in \operatorname{End}_{A}(R)$ is defined as
\[
[P, Q] = P \cdot Q - Q \cdot P
\]
\end{defn}

\begin{defn}
Let $D^0_{A}(R)$ be the set of maps in $\End_{A}(R)$ given by multiplication by an element of $R$. By identifying an endomorphism $\varphi$ with $\varphi(1)$, it is clear that $D_{A}^0(R)$ and  $R$ are isomorphic
as rings. For $i>0$, define the set of differential operators inductively by letting
\[
D_{A}^i(R) := \{\partial \in \End_{A}(R) \mid [\partial,a] \in D_{A}^{i-1}(R) \text{ for all } a \in R\}.
\]
Then the ring of differential operators on $R$ is 
\[
  D_{A}(R) := \bigcup_{i\geq 0} D_{A}^i(R).
\]
When $A$ is a field and $R$ is a polynomial ring over $A$, we write $D(R)$ to mean $D_A(R)$.
\end{defn}

\begin{defn}
\label{defoperator}
We say a differential operator $\partial \in D_A(R)$ has order $i$ if $\partial \in D_A^i(R) \setminus D_A^{i - 1}(R)$.
\end{defn}

This inductive description of differential operators can appear abstract. However, for many of our results in this paper, we focus on polynomial rings. In the case that $R$ is a polynomial ring over a field $\kk$ of characterstic $0$, the notion of differential operator agrees exactly with the more familiar notion of partial derivative: 

\begin{prop} \cite{MR1356713} \label{Weyl}
Let $R = \kk[x_1, \dots, x_n]$ where $\kk$ is a field of characteristic $0$.  Then $D(R)$ is the \emph{Weyl algebra} $\kk[x_1,\dots,x_n]\langle \partial_1, \dots, \partial_n \rangle$ where $\partial_i = \frac{\partial}{\partial x_i}$. 
\end{prop}

\begin{rmk}
The situation in positive characteristic is more subtle. If $R = \kk [x_1, \ldots, x_n]$, where $\kk$ is a field of characteristic $p>0$, then $D(R)$ includes operators of the form $\frac{1}{p^e!} \frac{\partial^{p^e}}{\partial x_i^{p^e}}$, for all $i$ and all nonnegative integers $e$; \emph{cf.~} \cite[Exercise 17.5]{srikanth_2007}. These are called \emph{divided power} operators and are not contained in the subring of $D(R)$ generated by operators of lower order. In particular, $D(R)$ is not a finitely generated algebra in this case.
\end{rmk}

We begin our study of the differential power by writing the following definition.

\begin{defn}[Definition 2.2, \cite{MR3779569}]\label{powerdef}
Let $R$ be an $A$-algebra and $I \subseteq R$ be an ideal. Define the $n^{\mathrm{th}}$ differential power of $I$ to be
\[
\diffpowerA{I}{n} = \{r \in R \mid \partial(r) \in I \ \text{for all} \ \partial \in D_A^{n-1}(R) \}
\]
When $A$ is a field and $R$ is a polynomial ring over $A$, we write $\diffpower{I}{n}$ to mean $\diffpowerA{I}{n}$.
\end{defn}

The differential power of an ideal is another ideal \cite[Proposition 3.8]{BRENNER2019106843}.  Futher, from the definition above, we can immediately see that $\diffpowerA{I}{n}\subseteq\diffpowerA{I}{m}$ for all $n \geq m$. Note that differential powers have the following property, which is useful when considering decompositions of ideals.

\begin{lemma}[Exercise 2.13 \cite{MR3779569}]\label{exercise2.13}
Let $\{I_\alpha\}_{\alpha\in S}$ be an indexed family of ideals. Then,
$$\bigcap_{\alpha\in S}\diffpowerA{I_\alpha}{n} =  \diffpowerA{\left(\bigcap_{\alpha\in S} I_\alpha \right)}{n}$$
for every positive integer $n$.
\end{lemma}

In Sections \ref{sec:monomialIdeals}, \ref{sec:principal} and \ref{sec:containment}, we work mostly in the setting when $R$ is a polynomial ring over a field. To determine if an element is in $\diffpower{I}{n}$, instead of checking all differential operators of a given order, it suffices to check only the differential operators of the form    $x_1^{\alpha_1} \dots x_d^{\alpha_d} \partial_1^{\beta_1} \dots \partial_d^{\beta_d}$, for all collections $\alpha_i, \beta_i \in \mathbb{Z}_{\geq 0}$. To simplify notation, for $\alpha \in \mathbb{Z}_{\geq 0}^d$, we will denote $x^\alpha \coloneqq x_1^{\alpha_1} \dots x_d^{\alpha_d}$ (where we note that some $\alpha_i$ may be $0$).  Similarly, we will denote $\partial^{\alpha} \coloneqq \partial_1^{\alpha_1} \dots \partial_d^{\alpha_d}$ where $\partial_i \coloneqq  \frac{\partial}{\partial x_i}$.  For any $\alpha \in \mathbb{Z}_{\geq 0}^d$, we will write $\supp(\alpha)$ for the \emph{support} of $\alpha$, namely $\supp{\alpha} = \{ i \mid \alpha_i \neq 0\}$.
For $\alpha,\beta \in \mathbb{Z}_{\geq 0}^d$, we will write $\alpha + \beta$ for the usual component-wise sum of tuples, namely $(\alpha + \beta)_i = \alpha_i + \beta_i$.  We will also write $\alpha\beta$ for the component-wise product, namely $(\alpha\beta)_i = \alpha_i\beta_i$, and finally, for  for $\alpha = (\alpha_1, \ldots, \alpha_d) \in \mathbb{Z}_{\geq 0}^d$, we write $|\alpha|$ to mean $\sum_{i=1}^d \alpha_i$.

In particular, we will use the following:

\begin{prop} [\cite{MR1356713} Chapter 1, Proposition 2.1] \label{weylalggens}
Let $R = \kk[x_1, \ldots, x_d]$ be a polynomial ring over a field $\kk$ of characteristic $0$. Then $D(R)$ is generated as a vector space over $\kk$ by
\[
\{ x^\alpha \partial^\beta \mid \alpha, \beta \in \mathbb{Z}_{\geq 0}^d\}
\]
where $x^\alpha\partial^\beta$ denotes $x_1^{\alpha_1}\cdots x_d^{\alpha_d}\frac{\partial^{\beta_1}}{\partial x_1^{\beta_1}}\cdots\frac{\partial^{\beta_d}}{\partial x_1^{\beta_d}}$
\end{prop}

From this, we can deduce the following corollary, which will be especially useful in computing differential powers.

\begin{cor}\label{differentiallemma}
Let $R = \kk[x_1, \ldots, x_d]$ be a polynomial ring over a field $\kk$ of characteristic $0$.  Then $y \in \diffpower{I}{n}$ if for every $\alpha \in \mathbb{Z}_{\geq 0}^d$ such that $|\alpha| < n$, we have $\partial^\alpha(y) \in I$.
\end{cor}

\begin{proof}
Let $\partial \in D^m(R)$ where $m < n$.  By Proposition \ref{weylalggens}, we can write $\partial = \sum_i r_i \partial^{\alpha_i}$ for some $r_i \in R$ and $\alpha_i \in \mathbb{Z}_{\geq 0}^d, |\alpha_i| < n$.
Then
$
\partial(y) = \left(\sum_i r_i \partial^{\alpha_i}\right)y = \sum_i r_i \left(\partial^{\alpha_i}(y)\right) \in I
$
since for each $i$, $\partial^{\alpha_i}(y) \in I$.  Thus $y\in \diffpower{I}{n}$.
\end{proof}

In the characteristic zero case, we also have the following result on multiplying differential powers.

\begin{lemma}\label{containmentdiffpower}
Let $R = \kk[x_1, \ldots, x_d]$ be a polynomial ring where $\kk$ is a field, and let $I \subseteq R$ be an ideal. If $\kk$ has characteristic zero, $\diffpower{I}{n}\cdot \diffpower{I}{m}\subseteq \diffpower{I}{m+n}$.
\end{lemma}

\begin{proof}
Let $s \in \diffpower{I}{n}$ and $t \in \diffpower{I}{m}$.  We want to show that $st \in \diffpower{I}{n+m}$.   Let $\alpha \in \mathbb{Z}_{\geq 0}^d$ such that $|\alpha| < m +n$.  By Corollary \ref{differentiallemma}, it suffices to check that $\partial^{\alpha}(st) \in I$.
By the General Leibniz rule,

$$\partial^{\alpha}(st) = \sum_{\omega \colon \omega \le \alpha} \binom{\alpha}{\omega} \partial^{\omega}(s) \partial^{\alpha - \omega} (t)$$

where

\[
\binom{\alpha}{\beta} = \binom{\alpha_1}{\beta_1} \binom{\alpha_2}{\beta_2} \cdots \binom{\alpha_d}{\beta_d}
\]
and we say $\omega \leq \alpha$ for $\omega \in \mathbb{Z}_{\geq 0}^d$ if $\omega_i \leq \alpha_i$ for all $i$ \cite{olver_2000}.

Note, we are only considering the order of the differential operators. In this case, we notice that if either $\partial^{\omega}(s)$ is in $I$ or $\partial^{\alpha - \omega}(t)$ is in $I$ for all $\omega \leq \alpha$ then $\partial^\alpha(st)\in I$. Equivalently, we need to show $|\alpha - \omega| \leq m - 1$ or $|\omega| \leq n - 1$ for each $\omega \leq \alpha$. Suppose $\omega \in \mathbb{Z}_{\geq 0}^d$ and $\omega \leq \alpha$.  We know $|\alpha - \omega| = |\alpha| - |\omega| \leq m + n - 1 - |\omega|$. If $|\alpha - \omega| > m - 1$, then $|\omega| \leq n - 1$. If $|\omega| > n - 1$, then $|\alpha - \omega| \leq m - 1$. So, either $|\alpha - \omega| \leq m - 1$ or $|\omega| \leq n - 1$. Thus we have shown $\partial_{\omega}(t) \in I$ or $\partial^{\alpha - \omega}(s)\in I$ for all $\omega \leq \alpha$. From this, we find that $\partial^{\alpha}(st) \in I$, so $st \in \diffpower{I}{n+m}$. Thus, $\diffpower{I}{n} \cdot \diffpower{I}{m} \subseteq \diffpower{I}{n+m}$.
\end{proof}

\begin{rmk}
 The conclusion of Lemma \ref{containmentdiffpower} is false in general---see \cite[Example 4.35]{BRENNER2019106843}.
\end{rmk}

\begin{lemma}\label{successivepowers}
If $R$ is a polynomial ring over a field of characteristic zero and $I$ is an ideal of $R$ then $\diffpower{\diffpower{I}{n}}{m} = \diffpower{I}{n + m - 1}$ for all $n, m>0$.
\end{lemma}
\begin{proof}
Let $r$ be an arbitrary element in $\diffpower{I}{n + m - 1}$. This means that $\partial(r) \in I$ for all differential operators $\partial$ with order less than or equal to $n + m - 2$. Let $\delta$ be a differential operator of order less than or equal to $n- 1$ and let $\eta$ be a differential operator of order less than or equal to $m -1$.  Then $\partial = \delta \cdot \eta$, has order less than $n + m - 2$, so $(\delta \cdot \eta)(r) = \partial(r) \in I$.  Since $\delta$ and $\eta$ were arbitrary, we have $r \in \diffpower{\diffpower{I}{n}}{m}$. 

For the reverse inclusion, let $r \in \diffpower{\diffpower{I}{n}}{m}$. By Corollary  \ref{differentiallemma}, it suffices to show that $\partial^\alpha(r) \in I$ for all $\alpha \in \mathbb{Z}_{\geq 0}^d$ with $|\alpha| \leq n + m - 2$. Given any such $\alpha$, we can write $\alpha = \beta + \gamma$ where $|\beta| \leq n-1$ and $|\gamma| \leq m-1$ (the entries of $\beta$ and $\gamma$ are allowed to be 0). In this case, we have
\[
  \partial^\alpha r = \partial^{\beta}( \partial^{\gamma} r) \in \partial^{\beta} \cdot \diffpower{I}{n} \subseteq I,
\]
as desired. 
\end{proof}

\section{Differential Powers of Monomial Ideals}
\label{sec:monomialIdeals}

In this section, let $R = \kk[x_1, \ldots, x_d]$ where $\kk$ is a field of characteristic 0. The goal of this section is to understand some basic properties of the differential power when the ideal is a monomial ideal. We show that if $I$ is a monomial ideal, then $\diffpower{I}{n}$ is a monomial ideal, and we provide an explicit description of its generators.

\begin{defn} \label{purePowers}
Let $R = \kk[x_1, \ldots, x_d]$ where $\kk$ is a field of characteristic 0 and $I$ an ideal of $R$. We say $I$ is \textit{generated by pure powers} if $I = (x_{i_1}^{\alpha_{i_1}}, \dots, x_{i_m}^{\alpha_{i_m}}) =  (x_i^{\alpha_i} \mid i \in \supp(\alpha))$ for some $\alpha \in \ZZ^d_{\geq 0}$.
\end{defn}

\begin{rmk}
Ideals generated by pure powers are sometimes referred to as ``M-irreducible" ideals in the literature, \emph{cf.~} \cite{monomialsbook}.
\end{rmk}

We recall some facts about monomial ideals.  First, every monomial ideal has a decomposition in terms of ideals which are pure powers of variables.  In particular,

\begin{thm}[Theorem 1.3.1 \cite{MR2724673}]\label{herzoghibi1.3.1}
Let $R = \kk[x_1, \ldots, x_d]$ where $\kk$ is a field of characteristic 0. Let $I$ be a monomial ideal. Then we can write $I = \bigcap_{i = 1}^{n} Q_i$ where each $Q_i$ is generated by pure powers of the variables. Moreover, an irredundant presentation of this form is unique up to reordering of the $Q_i$.
\end{thm}

In the particular case of square free monomial ideals, we see that this decomposition consists of ideals generated by variables.

\begin{cor}\label{1.3.1variant}
Let $R = \kk[x_1, \ldots, x_d]$ where $\kk$ is a field of characteristic 0. If $I$ is a square-free monomial ideal, then $I = \bigcap_{i = 1}^{n} Q_i$ where each $Q_i$ is of the form $(x_{i_1}, \ldots, x_{i_m})$.
\end{cor}

Furthermore, any finite intersection of monomial ideals is still a monomial ideal, and we can explicitly describe its generators.

\begin{defn}
Let $R = \kk[x_1, \ldots, x_d]$ where $\kk$ is a field of characteristic 0. Let $I$ be an ideal of $R$. The set of minimal generators of $I$ is denoted $G(I)$.
\end{defn}

\begin{prop}[Proposition 1.2.1 \cite{MR2724673}]\label{herzoghibi1.2.1}
Let $R = \kk[x_1, \ldots, x_d]$ where $\kk$ is a field of characteristic 0. Let $I$ and $J$ be monomial ideals. Then $I\cap J$ is a monomial ideal, and $\{\mathrm{lcm}(u,v)\colon u\in G(I),v\in G(J)\}$ is a set of generators of $I\cap J$.
\end{prop}

\begin{cor}
Let $R = \kk[x_1, \ldots, x_d]$ where $\kk$ is a field of characteristic 0. If $I_i$ is a monomial ideal and $n$ is a finite positive integer, then $\bigcap_{i = 1}^n I_i$ is monomial ideal.
\end{cor}

We will start by studying differential powers of ideals generated by pure powers of variables, which will allow us to compute differential powers of arbitrary monomial ideals. Consider $\alpha \in \mathbb{Z}_{\geq 0}^d$ and the ideal $I = (x_i^{\alpha_i} \mid i \in \supp(\alpha)) \subset R= \kk[x_1, \dots, x_d]$. We will give an explicit formula for the generators of $\diffpower{I}{n}$ for any positive integer $n$.  First, we show that if $I$ is a monomial ideal, then $\diffpower{I}{n}$ is a monomial ideal.

\begin{prop}\label{diffpowerismonomial}
Suppose $I$ is a monomial ideal, and $n \in \mathbb{N}$.  Then $\diffpower{I}{n}$ is also a monomial ideal
\end{prop}
\begin{proof}
Suppose that $f \in \diffpower{I}{n}$, where $f = \sum_{i=1}^m c_i x^{\alpha_i}$ where $\alpha_i \in \mathbb{Z}^d_{\geq 0}$ and $ 0 \neq c_i \in \kk$ for all $i$. We want to show that $x^{\alpha_i} \in \diffpower{I}{n}$ for all $i$. By Corollary \ref{differentiallemma}, it suffices to show that $\partial^\beta(x^{\alpha_i}) \in I$ for all $\beta \in \zdgeq$ with $|\beta| \leq n - 1$. If $\partial^\beta(x^{\alpha_i}) = 0$ then we're done; otherwise $\partial^\beta(x^{\alpha_i}) = cx^{\alpha_i - \beta}$ for some nonzero $c \in \kk$. In this case, $\partial^\beta (x^{\alpha_i})$ is $\kk$-linearly independent from $cx^{\alpha_j - \beta} $ for all $j \neq i$. In other words, $\partial^\beta(x^{\alpha_i})$ is a monomial appearing in the expansion of $\partial^\beta(f)$ with a non-zero coefficient. As $\partial^\beta(f) \in I$ and $I$ is a monomial ideal, this tells us that $\partial^\beta(x^{\alpha_i}) \in I$, as desired.
\end{proof}

\begin{thm}
\label{monomialformulaupdated}
Let $R = \kk[x_1, \ldots, x_d]$ where $\kk$ is a field of characteristic 0, and let $I = (x_i^{\alpha_i} \mid i \in \supp(\alpha))$, where $\alpha \in \mathbb{Z}_{\geq 0}^d$.  The differential power $\diffpower{I}{n}$ is generated by elements of the form
$x^\gamma = x_1^{\gamma_1}x_2^{\gamma_2}\ldots x_d^{\gamma_d}$ where $\gamma$ satisfies
\begin{enumerate}
    \item \label{comparison} $\gamma_i \geq \alpha_i$ for each $i \in \supp(\gamma)$, and
    \item \label{gammasAndalphas}
$\sum\limits_{\gamma_i \geq \alpha_i} \gamma_i = \sum\limits_{\gamma_i \geq \alpha_i} (\alpha_i-1) + n.$
\end{enumerate}
where the summations are taken over all indices $i$ such that $\gamma_i \geq \alpha_i$.
\end{thm}
\begin{proof}
Note that $\diffpower{I}{n}$ is a monomial ideal by Proposition \ref{diffpowerismonomial}. Suppose $\gamma \in \mathbb{Z}_{\geq 0}^d$.  We will first show that $x^{\gamma} \in \diffpower{I}{n}$ if and only if
\[
\sum_{\gamma_i \geq \alpha_i} \gamma_i \geq \sum_{\gamma_i \geq \alpha_i} (\alpha_i-1) + n.
\]

First suppose $\sum_{\gamma_i \geq \alpha_i} \gamma_i \geq \sum_{\gamma_i \geq \alpha_i} (\alpha_i-1) + n$.  Let $\beta \in \mathbb{Z}_{\geq 0}^d$ and $|\beta| \leq n-1$ so that $\partial^{\beta} \in D^{n-1}(R)$.  Denote $S = \supp(\gamma)$.
If there exists an index $j$ such that $\beta_j > \gamma_j$, then $\partial^\beta(x^{\gamma})= 0 \in I$. 
Now assume that $\beta_i \le \gamma_i$ for all $i$, which means $\beta_i = 0$ for all $i \notin S$. In other words, $\partial^\beta = \prod_{i\in S} \frac{\partial^{\beta_i}}{\partial x_i^{\beta_i}}$. Since $|\beta| \leq n-1$, we also have that $\sum_{i\in S} \beta_{i} \le n-1$. Note that
\[
\partial^{\beta}(x^{\gamma}) = c\prod_{i\in S} x_i^{\gamma_i - \beta_i}
\]
where $c$ is some constant in $\kk$. Since 
$$\sum_{i\in S} (\gamma_i - \beta_i) = \sum_{i\in S} \gamma_i - \sum_{i\in S} \beta_i
\geq \sum_{i\in S} (\alpha_i-1) + n - \sum_{i\in S} \beta_i \ge \sum_{i\in S} (\alpha_i-1) + 1,$$
then by the Pigeonhole Principle, there exists an index $l\in S$ such that $\gamma_l - \beta_l \ge \alpha_l$. This implies that $\partial^\beta(x^{\gamma} ) \in I$. By Corollary \ref{differentiallemma}, $x^{\gamma} \in \diffpower{I}{n}$.

Now suppose $\sum_{\gamma_i \geq \alpha_i} \gamma_i < \sum_{\gamma_i \geq \alpha_i} (\alpha_i-1) + n$.  Consider the differential operator $\partial^{\omega}$, where $\omega_i := \gamma_i - \alpha_i + 1$  whenever $\gamma_i \geq \alpha_i$ and $\omega_i := 0$ otherwise.  Let $S := \supp(\omega) = \{ i \mid \gamma_i \geq \alpha_i\}$.  First note that
\begin{align*}
\sum_{i=1}^d \omega_i = \sum_{i \in S} (\gamma_i - \alpha_i + 1) 
&= \sum_{i \in S} \gamma_i - \sum_{\gamma_i \geq \alpha_i} \alpha_i + |S| \\
&\leq \sum_{i \in S} (\alpha_i - 1) + n - 1 - \sum_{i \in S} \alpha_i + |S| \\
&= n - 1
\end{align*}
so that $\partial^{\omega} \in D^{n-1}(R)$.  Further, we have
\[
\partial^{\omega}(x^{\gamma}) = c\prod_{\gamma_i < \alpha_i} x^{\gamma_i} \prod_{\gamma_i \geq \alpha_i} x_i^{\gamma_i - (\gamma_i - \alpha_i + 1)}= c\prod_{\gamma_i < \alpha_i} x^{\gamma_i} \prod_{\gamma_i \geq \alpha_i} x_i^{\alpha_i - 1}
\]
for some $c \in \kk$, so that $\partial^{\omega}(x^{\gamma}) \notin I$, which means $x^{\gamma} \notin \diffpower{I}{n}$.

Thus, $x^{\gamma} \in \diffpower{I}{n}$ if and only if $\sum_{\gamma_i \geq \alpha_i} \gamma_i \geq \sum_{\gamma_i \geq \alpha_i} (\alpha_i-1) + n$.

Now note that if $\sum_{\gamma_i \geq \alpha_i} \gamma_i \geq \sum_{\gamma_i \geq \alpha_i} (\alpha_i-1) + n$, then we can write $x^{\gamma} = rx^{\gamma'}$ where $r = \prod_{\gamma_i < \alpha_i} x_i^{\gamma_i}$ and $x^{\gamma'} = \prod_{\gamma_i \geq \alpha_i} x_i^{\gamma_i}$, so that $\gamma'$ satisfies \eqref{comparison} by construction.  Thus, $\diffpower{I}{n}$ is generated by $x^\gamma$ satisfying \eqref{comparison} and $\sum_{\gamma_i \geq \alpha_i} \gamma_i \geq \sum_{\gamma_i \geq \alpha_i} (\alpha_i-1) + n$.  

Finally, suppose $x^{\gamma}$ satisfies \eqref{comparison} and $\sum_{\gamma_i \geq \alpha_i} \gamma_i > \sum_{\gamma_i \geq \alpha_i} (\alpha_i-1) + n$.  First suppose $|\supp(\gamma)| \leq n$.  Then, since
\[
\sum_{\gamma_i \geq \alpha_i} (\alpha_i - 1) + n = \sum_{\gamma_i \geq \alpha_i} \alpha_i - |\supp(\gamma)| + n \geq \sum_{\gamma_i \geq \alpha_i} \alpha_i
\]
we can write $x^{\gamma} = rx^{\gamma'}$ for some $r \in R$ so that $\gamma'$ satisfies \eqref{comparison} and \eqref{gammasAndalphas}.  
Otherwise, if $|\supp(\gamma)| > n$, let $\supp(\gamma) = \{i_1,\dots,i_m\}$ and notice that $m > n$.  Since $\gamma_i \geq \alpha_i$ for each $i \in \supp(\gamma)$, we have that
\[
\sum_{j=1}^n \gamma_{i_j} \geq \sum_{j=1}^n \alpha_{i_j} = \sum_{j=1}^n (\alpha_{i_j} -1) + n
\]
In particular, $x^{\gamma} = \prod_{j=1}^n x_{i_j}^{\gamma_{i_j}}\prod_{j=n+1}^m x_{i_j}^{\gamma_{i_j}}$ and if we let $x^{\gamma'} := \prod_{j=1}^n x_{i_j}^{\gamma_{i_j}}$, then $\gamma'$ satisfies \eqref{comparison}, $\sum_{\gamma_i' \geq \alpha_i} \gamma_i' \geq \sum_{\gamma_i' \geq \alpha_i} (\alpha_i-1) + n$, and $|\supp(\gamma')| = n$, so that by the argument above, we can write $x^{\gamma} = rx^{\gamma''}$ for some $r \in R$ and $\gamma''$ satisfying \eqref{comparison} and \eqref{gammasAndalphas}.
\end{proof}

\begin{ex}
Let $R = \CC[x, y]$ and $I = (x^2, y^3)$. By Theorem \ref{monomialformulaupdated}, when $n = 2$, the generator with only the $x$ term is $x^{(2-1)+2} = x^3$, the generator with only the $y$ term is $y^4$, and the generator with both $x$ and $y$ terms is $x^2y^3$ where $2+3 = (2-1)+(3-1)+2$. So we have $\diffpower{I}{2} = (x^3,y^4,x^2y^3)$.\\
When $n = 3$, the generator with only the $x$ term is $x^4$, the generator with only the $y$ term is $y^5$, and the generators with both $x$ and $y$ terms are $x^3y^3$ and $x^2y^4$, we can check $3+3 = 2+4 = (2-1)+(3-1)+3$. So we have $\diffpower{I}{3} = (x^4,y^5,x^3y^3,x^2y^4)$.
\end{ex}

Since differential powers commute with intersection, by Lemma \ref{exercise2.13}, we can use Theorem \ref{herzoghibi1.3.1} and Theorem \ref{monomialformulaupdated} to compute differential powers of any monomial ideal.  We illustrate this by giving a formula for the differential powers of principal monomial ideals below.

Since Lemma \ref{exercise2.13} tells us the differential power respects intersections, we are able to compute the $n$-th differential power of any monomial ideal by breaking it up into an intersection of ideals generated by pure powers of variables. When $I$ is generated by a single monomial, we can explicitly write the generators of $\diffpower{I}{n}$ for all $n$.

\begin{prop}\label{SingleMonomialDiffPower}
Let $R = \kk[x_1, \ldots, x_d]$ where $\kk$ is a field of characteristic 0, and let $\alpha \in \mathbb{Z}_{\geq 0}^d$. Consider the principal ideal $I =  ( x^{\alpha})\subset R$.
Then 
\[\diffpower{I}{n} = \left(\prod_{i\in\supp(\alpha)}x_i^{\alpha_i + n - 1}\right).\]
\end{prop}
\begin{proof}
First notice that $I = \bigcap_{i \in \supp(\alpha)} (x_i^{\alpha_i})$. From Lemma \ref{exercise2.13}, we know that differential powers respect finite intersections. Thus,
\[
\diffpower{I}{n} =\bigcap_{i \in \supp(\alpha)} \diffpower{(x_i^{\alpha_i})}{n}
\]
We know how to compute each of these differential powers by Theorem \ref{monomialformulaupdated}. This gives us
\[
\diffpower{I}{n} = \bigcap_{i\in\supp(\alpha)} (x_i^{\alpha_i - 1 + n})
\]
Finally, we rewrite this intersection as a single ideal
\[
\diffpower{I}{n} = \left( \prod_{i \in \supp(\alpha)} x_i^{\alpha_i +n - 1}\right) \qedhere
\] 
\end{proof}

More generally, we can compute the differential power of an arbitrary monomial ideal using Theorem \ref{monomialformulaupdated}.

\begin{ex}
Let $R = \CC[x,y,z]$ and $I = (xy,z^2) = (x,z^2)\cap (y,z^2)$. If we want to calculate $\diffpower{I}{3}$, it is sufficient to calculate $\diffpower{(x,z^2)}{3}\cap\diffpower{(x,z^2)}{3} = (x^3,x^2z^2,xz^3,z^4)\cap (y^3,y^2z^2,yz^3,z^4) = (x^3y^3,x^2y^2z^2,xyz^3,z^4)$.
\end{ex}

\section{Differential Powers which are Eventually Principal}
\label{sec:principal}

In this section, we will study when $\diffpower{I}{n}$ is principal for some $n \gg 0$.  We will discuss only the case that $R = \kk[x,y]$ and $R =\kk[x,y,z]$, starting with the two variable case.
Notice that in this case, we can create a graphical representation for every monomial ideal in $R$ using an integer lattice where the $x$-axis is the power of $x$ and $y$-axis is the power of $y$. For example, the point $(2,3)$ would represent the set of monomials $cx^2y^3$ where $c\in \kk^\times$ and the ideal $I = (xy^2,x^3)$ would be all the points inside the shaded area. This shaded region is called the \emph{exponent set} of $I$; see \cite[\S 1.4]{HunekeSwansonIntegral}.
\begin{figure}[H]
    \centering
    \input{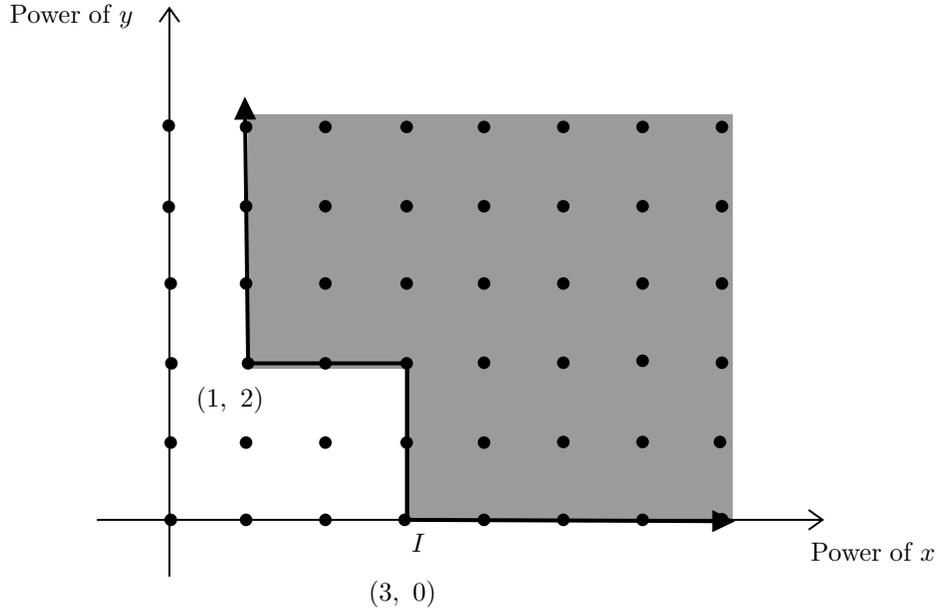}
    \caption{The exponent set of $I = (xy^2,x^3)$}
    \label{fig:IntegerLatticExJustIdeal}
\end{figure}

Taking the partial derivative $\partial/ \partial y$ would be a translation down in the $y$ direction, and taking the partial derivative $\partial/ \partial x$ would be a translation left in the $x$ direction. Now we can also represent the differential powers of monomial ideals graphically. For example, the 2nd and 3rd differential power of the ideal $I = (xy^2,x^3)$ can be represented as
\begin{figure}[H]
    \centering
    \input{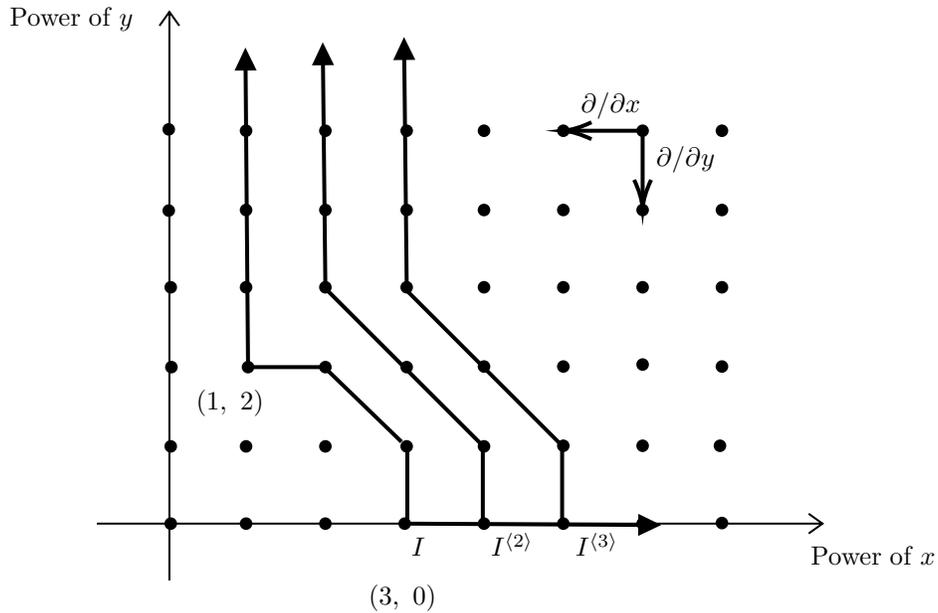}
    \caption{The exponent set of $I = (xy^2,x^3)$ and its differential powers}
    \label{fig:IntegerLatticeEx}
\end{figure}

We now show that for all monomial ideals $I \subseteq \kk[x,y]$ with $\sqrt{I} = (xy)$, there exists a positive integer $N$ such that $\diffpower{I}{N}$ is principal. The case where $I$ is itself principal is covered by Proposition \ref{SingleMonomialDiffPower}, so we may assume that $I$ has at least two generators. First we fix some notation. 

\begin{setting}\label{2variablesetting}
Let $I$ be a non-principal monomial ideal in $R = \kk[x,y]$. We let $\{ x^{\beta_{11}}y^{\beta_{12}},\ldots,x^{\beta_{m1}}y^{\beta_{m2}}\}$ be the set of minimal generators of $I$, where $m>1$. Notice that we can rearrange the generators of $I$ to satisfy the inequality $\beta_{11} < \cdots < \beta_{m1}$ because $\beta_{k1} = \beta_{(k+1)1}$ implies that one of  $x^{\beta_{k1}}y^{\beta_{k2}}$ and $x^{\beta_{(k+1)1}}y^{\beta_{(k+1)2}}$ divides the other. This inequality implies another inequality $\beta_{12} > \cdots > \beta_{m2}$ because if $\beta_{l2} \le \beta_{(l+1)2}$ and $\beta_{l1} < \beta_{(l+1)1}$, then  $x^{\beta_{(l+1)1}}y^{\beta_{(l+1)2}}$ is divisible by $x^{\beta_{l1}}y^{\beta_{l2}}$. With our generators of $I$ labelled in this way, we fix $i$ to be an index such that $\beta_{(i+1)1}+\beta_{i2} = \max\{\beta_{(j+1)1}+\beta_{j2} \mid 1\le j < m\}$
\end{setting}

\begin{lemma}\label{sumOfPowerInTheIdeal} Let $I$ and $i$ be as in setting \ref{2variablesetting}. If $a+b \ge \beta_{(i+1)1}+\beta_{i2}-1$ and $a \ge \beta_{11}$, $b \ge \beta_{m2}$, then $x^ay^b \in I$.
\end{lemma}
\begin{proof}
Suppose there exists such $a,b$ satisfying the condition, and $x^ay^b\not\in I$. Since $a \ge \beta_{11}$ and $b \ge \beta_{m2}$, then there exists an index $k$ such that $\beta_{k1} \le a < \beta_{(k+1)1}$.  Then since $x^ay^b \notin I$, we know that $x^{\beta_{k1}}y^{\beta_{k2}}$ does not divide $x^ay^b$ so that $\beta_{k2} > b$.  However, since $\beta_{k1},\beta_{(k+1)1},\beta_{k2},a,b$ are positive integers, then we have $\beta_{(k+1)1} \ge a + 1$ and $\beta_{k2} \ge b+1$. This implies that $\beta_{(k+1)} + \beta_{k2} \ge a+b+2 \ge \beta_{(i+1)1}+\beta_{i2} + 1$ which contradicts the assumption that $\beta_{(i+1)1}+\beta_{i2} = \max\{\beta_{(j+1)1}+\beta_{j2} \mid 1\le j < m\}$. Thus, $x^ay^b\in I$.
\end{proof}

\begin{thm}\label{principal2d}
Let $I$ and $i$ be as in setting \ref{2variablesetting}. Further, suppose that $\beta_{11}>0$ and $\beta_{m2}>0$. Set $N = \beta_{(i+1)1} + \beta_{i2} - \beta_{m2} - \beta_{11}$. Then $\diffpower{I}{n}$ is principal if and only if $n \geq N$.
\end{thm}

\begin{proof} By Proposition \ref{SingleMonomialDiffPower}, the differential powers of a principal monomial ideal are principal monomial ideals. Thus, by Lemma \ref{successivepowers}, it suffices to show that $\diffpower{I}{N}$ is principal and that $\diffpower{I}{N-1}$ is not. 

To see that $\diffpower{I}{N}$ is principal, let $a = \beta_{(i+1)1}+\beta_{i2}-\beta_{m2}-1$, and $b =\beta_{(i+1)1}+\beta_{i2}-\beta_{11}-1$. We claim that $\diffpower{I}{N} = (x^ay^b)$. To see this, it suffices to show that $x^ay^b$ is an element of $\diffpower{I}{N}$ and that $x^{a-1}y^{e}$ and $x^{f}y^{b-1}$ are not elements of $\diffpower{I}{N}$ for all positive integers $e,f$. To show that $x^ay^b\in \diffpower{I}{N}$, it is sufficient to check that if $\omega = (d,N-1-d)$ for some $0 \leq d \leq N-1$, then $\partial^{\omega}(x^ay^b) \in I$.
Note that
\begin{align*}
    \partial^{\omega}(x^ay^b) &= \frac{\partial^{N-1}}{\partial x^{d} \partial y^{N-1-d}}(x^ay^b)\\
    &= cx^{a-d}y^{b-N+1+d}\\
    &= cx^{a-d}y^{\beta_{m2}+d}
\end{align*}
for some $c$ in $\kk^\times$.
By Lemma \ref{sumOfPowerInTheIdeal}, since $(a-d) + (\beta_{m2} + d) = a + \beta_{m2} = \beta_{(i+1)1}+\beta_{i2} - 1$ and $a - d \ge a-N+1 = \beta_{11}$, and $\beta_{m2} + d \ge \beta_{m2}$, then we know that $x^{a-d}y^{\beta_{m2}+d} \in I$, so that $x^ay^b \in \diffpower{I}{n}$

Now, let $e$ and $f$ be nonnegative integers. Because $\beta_{11}>0$ and $\beta_{m2}>0$, we have $\frac{\partial^{N-1}}{\partial x^{N-1}}(x^{a-1}y^{e}) = c_1x^{\beta_{11} -1}y^e \not\in I$ and $\frac{\partial^{N-1}}{\partial y^{N-1}}(x^{f}y^{b-1}) = c_2x^fy^{\beta_{m2} - 1} \not\in I$, for some $c_1, c_2 \in \kk^\times$.  In particular $x^{a-1}y^{e},x^{f}y^{b-1}\not\in \diffpower{I}{N}$, for any $e,f > 0$ so that $\diffpower{I}{N} = (x^{a}y^{b})$. 

Finally, we show that $\diffpower{I}{N-1}$ is not principal. Letting $a$ and $b$ be as before, we can show that $x^{a-1}y^b,x^ay^{b-1}\in\diffpower{I}{N-1}$ and $x^{a-1}y^{b-1}\not\in \diffpower{I}{N-1}$. By Corollary \ref{differentiallemma}, we only need to consider differential operators of the form $\partial^{\omega} = \frac{\partial^{|\omega|}}{\partial x^{\omega_1}\partial y^{\omega_2}}$ where $\omega \in \mathbb{Z}_{\geq 0}^2$ and $|\omega| \le N-2$. 
Note that
\begin{align*}
    a-1-\omega_1 + b - \omega_2 &\ge a+b-1-N+2\\
    &= \beta_{(i+1)1}+\beta_{i2}-1.
\end{align*}
Then by Lemma \ref{sumOfPowerInTheIdeal}, $\partial^{\omega}(x^{a-1}y^b)\in I$ which implies that $x^{a-1}y^b\in \diffpower{I}{N-1}$. Similarly,
\begin{align*}
    a-\omega_1 + b -1 - \omega_2 &\ge a+b-1-N+2\\
    &= \beta_{(i+1)1}+\beta_{i2}-1,
\end{align*}
so that by Lemma \ref{sumOfPowerInTheIdeal}, $\partial^{\omega}(x^{a}y^{b-1})\in I$ which implies that $x^{a}y^{b-1}\in \diffpower{I}{N-1}$. 
Since 
\[\frac{\partial^{N-2}}{\partial x^{\beta_{i2}-\beta_{m2}-1} \partial y^{\beta_{(i+1)1}-\beta_{11}-1}}(x^{a-1}y^{b-1}) = cx^{\beta_{(i+1)1}-1}y^{\beta_{i2}-1}\not\in I,
\]
(for some $c \in \kk$), we have $x^{a-1}y^{b-1}\not\in \diffpower{I}{N-1}$, which proves our claim.
\end{proof}

\begin{smallex}
Let $I= (x^2y^5,x^4y^3,x^5y)$, then we know that $\diffpower{I}{N}$ is principal where $N = 5+4-1-2 = 6$. This means that $\diffpower{I}{6} = (x^7y^6)$. We can confirm this by working with exponent sets, as in Figure \ref{fig:princpalDiffPower}.
\end{smallex}

\begin{figure}[H]
    \centering
    \input{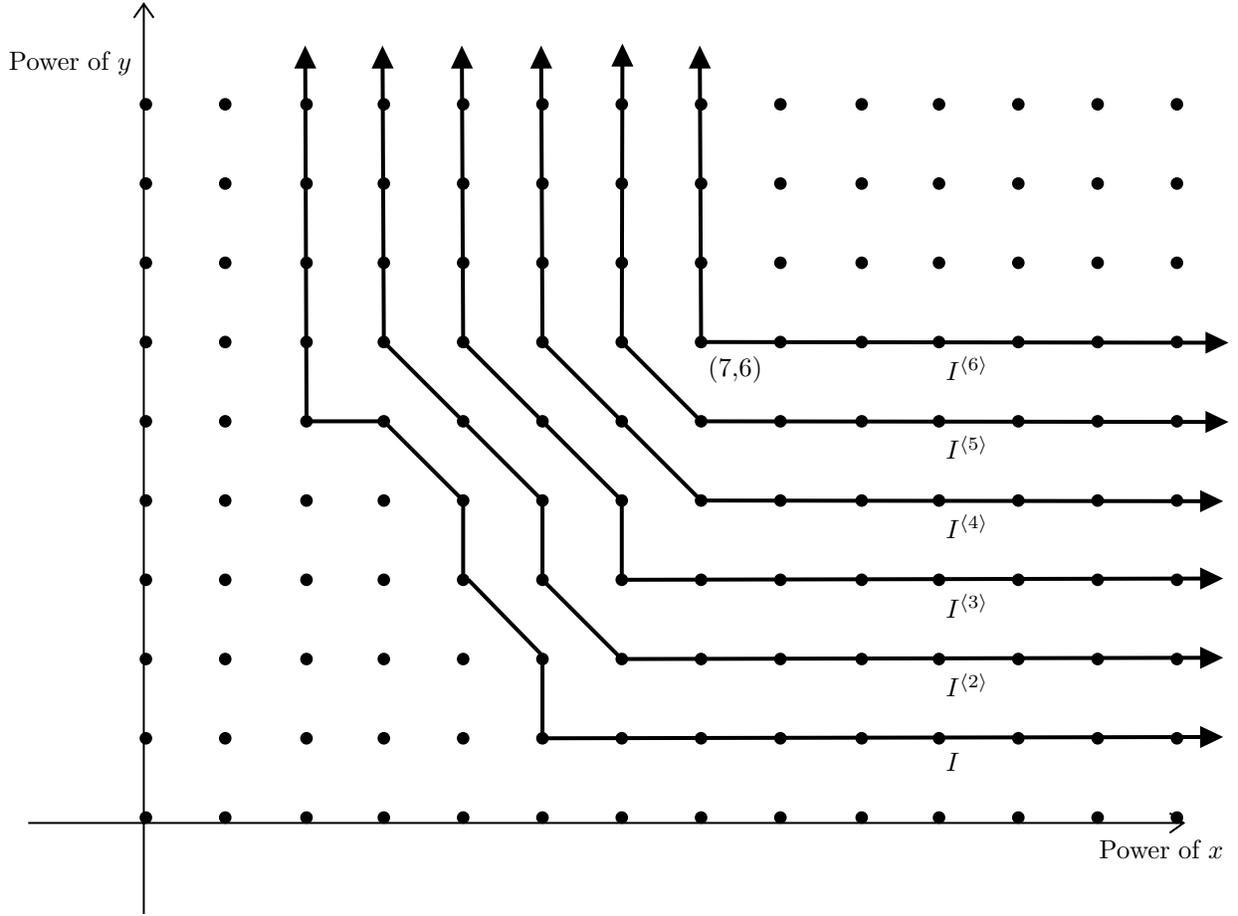}    
    \caption{The exponent sets of $I= (x^2y^5,x^4y^3,x^5y)$ and its differential powers.}
    \label{fig:princpalDiffPower}
\end{figure}

We can also prove an analog to Theorem \ref{principal2d} in $3$ variables.

\begin{thm}\label{principal3d}
Let $R = \kk[x,y,z]$, and let $I := (x^{\beta_{11}}y^{\beta_{12}}z^{\beta_{13}},\ldots,x^{\beta_{m1}}y^{\beta_{m2}}z^{\beta_{m3}}) \subset R$ be an ideal such that $\sqrt{I} = (xyz)$.  Then there exists a positive integer $N$ such that $\diffpower{I}{N}$ is principal.
\end{thm}

\begin{proof}
Let
\begin{align*}
    \beta_{xmax} = \max\{\beta_{i1}\mid i\in [m]\} &&
    \beta_{ymax} = \max\{\beta_{i2}\mid i\in [m]\} &&
    \beta_{zmax} = \max\{\beta_{i3}\mid i\in [m]\} \\
    \beta_{xmin} = \min\{\beta_{i1}\mid i\in [m]\} &&
    \beta_{ymin} = \min\{\beta_{i2}\mid i\in [m]\} &&
    \beta_{zmin} = \min\{\beta_{i3}\mid i\in [m]\}
\end{align*}

Define 
\begin{align*}
    N_{xy} &= \beta_{xmax} + \beta_{ymax} - \beta_{xmin} - \beta_{ymin} \\
    N_{yz} &= \beta_{ymax} + \beta_{zmax} - \beta_{ymin} - \beta_{zmin} \\
    N_{xz} &= \beta_{xmax} + \beta_{zmax} - \beta_{xmin} - \beta_{zmin}
\end{align*}
Let $N = \max\{N_{xy},N_{yz},N_{xz}\}$ and
\begin{align*}
    a = N-1 + \beta_{xmin} && b = N-1 + \beta_{ymin} && c = N-1 + \beta_{zmin}
\end{align*}

We will show that $\diffpower{I}{N} = (x^ay^bz^c)$.
It suffices to show that $x^ay^bz^c \in \diffpower{I}{N}$ and for any positive integers $d,e$, we have that $x^{a-1}y^{d}z^{e}$, $x^{d}y^{b-1}z^{e}$, and  $x^{d}y^{e}z^{c-1}$ are not in $\diffpower{I}{N}$. Since $\frac{\partial^{N-1}}{\partial x^{N-1}}x^ay^bz^c = C_1x^{\beta_{xmin}}y^bz^c \not\in \kk$, $\frac{\partial^{N-1}}{\partial y^{N-1}}x^ay^bz^c = C_2x^ay^{\beta_{ymin}}z^c \not\in \kk$, and $\frac{\partial^{N-1}}{\partial z^{N-1}}x^ay^bz^c = C_3x^ay^bz^{\beta_{zmin}} \not\in \kk$, then it is sufficient to check all of the differential operators in the form $\frac{\partial^{N-1}}{\partial x^{\omega_1} \partial y^{\omega_2} \partial z^{\omega_3}}$ where $\omega_1+\omega_2+\omega_3 = N - 1$.
Let $\omega \in \mathbb{Z}_{\geq 0}^3$ and $|\omega| = N-1$.  

For contradiction, suppose that $x^{a-\omega_1}y^{b-\omega_2}z^{c-\omega_3}\not\in I$ which is equivalent in $x^ay^bz^c\not\in \diffpower{I}{N}$. Since we know
\begin{align*}
    a - \omega_1 &\geq \beta_{xmin}\\
    b - \omega_2 &\geq \beta_{ymin},\textrm{ and }\\
    c - \omega_3 &\geq \beta_{zmin},
\end{align*}
we see that any two of 
\begin{align*}
    a-\omega_1 &\ge \beta_{xmax},\\
    b-\omega_2 &\ge \beta_{ymax},\textrm{ and }\\
    c-\omega_3 &\ge \beta_{zmax}
\end{align*}
being true implies $x^{a-\omega_1}y^{b-\omega_2}z^{c-\omega_3}\in I$. Thus, by the contrapositive, $x^{a-\omega_1}y^{b-\omega_2}z^{c-\omega_3}\not\in I$ implies that two of $a-\omega_1 < \beta_{xmax}, b-\omega_2 < \beta_{ymax},$ and $c-\omega_3 < \beta_{zmax}$ are true. Without loss of generality, suppose that $a-\omega_1 < \beta_{xmax}$ and $ b-\omega_2 < \beta_{ymax}$. This implies that $a-\omega_1 \le \beta_{xmax} -1$ and $ b-\omega_2 \le \beta_{ymax}-1$. This means that $a-\omega_1 + b-\omega_2 \le \beta_{xmax} + \beta_{ymax} - 2$. However, we know that
\begin{align*}
    a-\omega_1 + b-\omega_2 &\ge a + b-\omega_1-\omega_2-\omega_3\\
    &= a+b -N+1\\
    &= 2N - 2 + \beta_{xmin} + \beta_{ymin} - N + 1\\
    &= N -  1 + \beta_{xmin} + \beta_{ymin}\\
    &\ge \beta_{xmax} + \beta_{ymax} - 1.
\end{align*}
This is a contradiction, so $x^{a-\omega_1}y^{b-\omega_2}z^{c-\omega_3} \in I$ which implies that $x^ay^bz^c\in \diffpower{I}{N}$.

Now, because $\sqrt{I} = (xyz)$, we know that $\beta_{xmin}, \beta_{ymin},$ and $\beta_{zmin}$ are positive integers. Thus we see that 
\begin{align*}
    \frac{\partial^{N-1}}{\partial x^{N-1}}x^{a-1}y^{d}z^{e} &= C_1x^{\beta_{xmin}-1}y^{d}z^{e} \not\in I,\\ \frac{\partial^{N-1}}{\partial y^{N-1}}x^{d}y^{b-1}z^{e} &= C_2x^dy^{\beta_{ymin}-1}z^e \not\in I, \textrm{ and }\\ \frac{\partial^{N-1}}{\partial z^{N-1}}x^{d}y^{e}z^{c-1} &= C_3x^dy^ez^{\beta_{zmin}-1} \not\in I
\end{align*}
for all inegers $d, e \geq 0$. Thus we know that $\diffpower{I}{N} = (x^a y^b z^c)$.
\end{proof}

Unlike Theorem \ref{principal2d}, the $N$ found in Theorem \ref{principal3d} is not necessarily optimal. This prompts the following question.

\begin{ques}\label{eventuallyPrincipleQuest}
Given a monomial ideal $I$ in a ring $\kk[x_1, \ldots, x_d]$ with $\sqrt{I} = (x_1\cdots x_d)$, is there an $n$ such that $\diffpower{I}{n}$ is principal? What is the minimal such $n$?
\end{ques}

We conclude this subsection with a list of examples with omitted computations.
\begin{smallex}
For each $I \subseteq \kk[x,y,z]$ below, we list the $N$ produced by Theorem \ref{principal3d}, so that $\diffpower{I}{N}$ is principle.  We also list $N_{\min} := \min \{ n \mid \diffpower{I}{n}$ is principle $\}$.

\begin{center}
\begin{tabular}{|c|c|c|c|} \hline
Ideal $I$ & $N$ & $N_{\min}$ & $I^{\langle N_{\min} \rangle }$  \\ \hline
$(x^5 y^4 z, x^2 y z^{10}, x y^7 z^3)$    & 15  & 12 & $(x^{12}y^{12}z^{12})$ \\ \hline 
$(x^4 y z, x y^4 z, x y z^4)$ & 6 & 5 & $(x^{5}y^{5}z^{5})$ \\ \hline 
$(x^2 y^7 z^3, x^7 y^2 z^5, x y^5 z^7)$ & 11 & 9 & $(x^{9}y^{10}z^{11})$ \\
\hline 
$(x^7 y^9 z^3, x^5 y^5 z^7, x^{10} y^4 z^8)$ & 10 & 9 &  $(x^{13}y^{12}z^{11})$ \\ \hline
\end{tabular}
\end{center}
\end{smallex}

\section{The Containment Problem}
\label{sec:containment}
Inspired by a rich body of work comparing symbolic and ordinary powers of ideals, we ask the following questions: given an ideal $I \subseteq \kk[x_1, \ldots, x_d]$, for which integers $a$ and $b$ do we have $\diffpower{I}{a}\subseteq I^b$\ ? For which integers $a$ and $b$ do we have the reverse inclusion, $\diffpower{I}{a}\supseteq I^b$\ ? Since we have an explicit way (Theorem \ref{monomialformulaupdated} and Proposition \ref{SingleMonomialDiffPower}) to write out the generators of the $n$-th differential power of monomial ideals generated by pure powers of variables and of principle monomial ideals, we focus on these two cases.

First, we want to study the conditions in which $I^a\subseteq \diffpower{I}{b}$. The following holds in great generality:

\begin{prop}[Proposition 2.5, \cite{MR3779569}]\label{normalpowerContainment}
Let $R$ be a finitely generated $A$-algebra and $I$ an ideal of $R$. Then $I^n\subseteq \diffpowerA{I}{n}$.
\end{prop}

We aim to find sharper variants of this containment when $I$ is a monomial ideal in a polynomial ring of characteristic 0. One may hope for a $c > 1$ such that $I^n \subseteq \diffpower{I}{cn}$ for all $I$ and all $n$, but this is not possible: given such a $c$, consider the ideal $I = (x_1^{c+1})$. By Theorem \ref{monomialformulaupdated}, $\diffpower{I}{cn}$ is generated by the element $x_1^{c+cn}$. We know that $I^n$ is generated by $x_1^{cn + n}$. Since $I^n\subseteq \diffpower{I}{cn}$ implies that $cn + n \ge c + cn$, then we have $n \ge c$ for all $n\in \ZZ_{> 0}$ which implies that $c = 1$.

Thus, we seek containments of the form $I^n \subseteq \diffpower{I}{n+c}$ instead. We have the following:

\begin{prop}
$R = \kk[x_1, \ldots, x_d]$, where $\kk$ is a field of characteristic $0$, and let $I = (x_i^{\alpha_i} \mid i \in \supp(\alpha))$  for some $\alpha \in \mathbb{Z}_{\geq 0}^d$. Fix an arbitrary positive integer $n$. Then 
\[
  I^n \subseteq  \diffpower{I}{n+c},
\] 
where  $c = \min\{\sum_{i \in \supp(\omega)} (\alpha_i - 1)(\omega_i -1) \mid \omega \in \mathbb{Z}^d_{\ge 0}, \; \supp(\omega) \subset \supp(\alpha), \text{ and }  |\omega| = n\}$.
\end{prop}

\begin{proof}
We can write the set of generators of $I^n$ as \[\{x_1^{\alpha_1\omega_1}\cdots x_d^{\alpha_d\omega_d}\mid \omega \in \mathbb{Z}_{\geq 0}^d, \; \supp(\omega) \subset \supp(\alpha), \text{ and } |\omega| = n\}.\]
Let $x^{\alpha\omega} := x_1^{\alpha_1\omega_1}\cdots x_d^{\alpha_d\omega_d}$ be an arbitrary one of these monomial generators of $I^n$, and let $S := \supp(\omega)$. Note that $\omega_i\alpha_i \geq \alpha_i$ for each $i \in \supp(\omega_i\alpha_i) = \supp(\omega_i) = S$.  Also,
\begin{align*}
    \sum_{\omega_i\alpha_i \geq \alpha_i} \alpha_i \omega_i
    &=\sum_{i \in S} \alpha_i\omega_i\\
    &= \sum_{i\in S}(\alpha_i - 1)\omega_i + \sum_{i\in S}\omega_i\\
    &= \sum_{i\in S}(\alpha_i - 1)\omega_i + n\\
    &= \sum_{i\in S}(\alpha_i - 1) + \sum_{i\in S}(\alpha_i -1 )(\omega_i - 1) + n \\
    &\geq \sum_{i\in S}(\alpha_i - 1) + c  +n.
\end{align*}
Thus, by Theorem \ref{monomialformulaupdated}, $x^{\alpha\omega} \in \diffpower{I}{n+c}$, which means that $I^n \subseteq  \diffpower{I}{n+c}$.
\end{proof}

From \cite{MR1826369}, we have the containment is $I^{(dn)}\subseteq I^n$ for all positive integers $n$ and radical ideals $I$. Our goal is to find a constant $c$ that depends only on the ring $R$, such that $\diffpower{I}{cn}\subseteq I^n$ for some certain type of ideal $I$. However, such a uniform containment result is not possible. In fact, the situation is hopeless even if we replace $cn$ with any polynomial function of $n$:

\begin{prop}\label{prop:noContainment}
Let $R = \kk[x_1, \ldots, x_d]$ where $\kk$ is a field of characteristic 0. There does not exist a polynomial function $p: \mathbb N \to \mathbb N$ such that $\diffpower{I}{p(n)}\subseteq I^n$ for all $n\in \ZZ_{\ge 0}$.
\end{prop}
\begin{proof} Let $p(n)$ be a polynomial function as in the theorem statement. Since $\diffpower{I}{p(n)+1} \subseteq \diffpower{I}{p(n)}$ for all $n$, we may assume that $p(n) \geq 1$ for all $n$. Let $I = (x_1^c)$ where $c = p(2)$. Then by Theorem \ref{monomialformulaupdated}, $\diffpower{I}{p(n)} = (x_1^{c-1 + p(n)})$. Since $I^n = (x_1^{cn})$, we have for all $n \in \mathbb{N}$,
\begin{align*}
    c - 1 + p(n) &\ge cn \\
    p(n) &\ge c (n-1) + 1.
\end{align*}
However, when $n = 2$ this inequality gives $p(2) \ge p(2) + 1$ which leads to a contradiction. Therefore, there does not exist a polynomial $p(n)$ such that $\diffpower{I}{p(n)}\subseteq I^n$ for all $n\in \ZZ_{\ge 0}$.
\end{proof}

Thus, we seek containments of the form $\diffpower{I}{cn} \subseteq I^n$ where $c$ may depend on $I$. First, we want to simplify the condition by letting $I$ be any irreducible monomial ideal, and we want to find a constant $c$ that depends on $n$ and see if the sequence of $c$ on $n$ converges.

\begin{thm}\label{irreduciblecontainment}
Let $R = \kk[x_1, \ldots, x_d]$ where $\kk$ is a field of characteristic 0, and let $I = (x_i^{\alpha_i} \mid i \in \supp(\alpha))$ for some $\alpha \in \mathbb{Z}_{\geq 0}^d$. If $c$ is a rational number with $c \geq \max\{\alpha_1,\ldots, \alpha_d\}$, $c> |\supp(\alpha)|$, and $cn \in \mathbb Z$, then $\diffpower{I}{cn} \subseteq I^n$.
\end{thm}
\begin{proof}
Let $c$ be as in the theorem statement and let $x^\gamma$ be one of the monomial generators of $\diffpower{I}{cn}$ from Theorem \ref{monomialformulaupdated}.  Then we know
\begin{enumerate}
    \item $\gamma_i \geq \alpha_i$ for each $i \in \supp(\gamma)$, and \label{cond1}
    \item $\sum\limits_{\gamma_i \geq \alpha_i} \gamma_i = \sum\limits_{\gamma_i \geq \alpha_i} (\alpha_i-1) + cn$ \label{cond2}
\end{enumerate}
We also know that $I^n$ is generated by \[\{x^{\alpha\omega} \mid \omega \in \mathbb{Z}_{\geq 0}^d, \; \supp(\omega) \subset \supp(\alpha), \text{ and } |\omega| = n\}.\]

We want to show that $x^{\gamma} \in I^n$. In other words, we want to show that there exists $\omega\in \ZZ_{\ge 0}^d$ such that $\supp(\omega) \subseteq \supp(\alpha)$, $|\omega| \geq n$, and $\gamma_i \ge \alpha_i \omega_i$ for all $i \in [d]$.  Let $S := \supp(\gamma)$. Let $\delta_i = \gamma_i - \alpha_i$ for every $i\in S$ and $\delta_i =0$ otherwise. 
Since $\gamma_i \ge \alpha_i$ for all $i\in S$, then $\delta_i \ge 0$ for all $i\in S$. 
Also, 
\begin{align*}
\sum_{i \in S} \delta_i &= \sum_{i \in S} (\gamma_i - \alpha_i) \\
&= \sum_{i\in S}(\alpha_i - 1) + cn - \sum_{i \in S} \alpha_i \\
&= cn - |S|
\end{align*}
Let $\omega_i = \lceil \frac{\delta_i}{c} \rceil$.  Then $\gamma_i = \alpha_i(1 + \frac{\delta_i}{\alpha_i}) \geq \alpha_i(1 + \frac{\delta_i}{c}) \geq \alpha_i \omega_i$

Also,
$$\sum_{i\in S} \omega_i = \sum_{i\in S}\left\lceil \frac{\delta_i}{c} \right\rceil \ge \left\lceil\sum_{i\in S}\frac{\delta_i}{c} \right\rceil = \left\lceil n - \frac{|S|}{c} \right\rceil = n$$
because $|S|< c$. We have shown that each generator of $\diffpower{I}{cn}$ is in $I^n$. Thus, $\diffpower{I}{cn}\subseteq I^n$.
\end{proof}

We can also find such a $c$ for principal monomial ideals.

\begin{prop}\label{SingleMonomialContainmentDependsonN}
Let $R = \kk[x_1, \ldots, x_d]$ where $\kk$ is a field of characteristic 0, and let $I = (x^{\alpha})$ for some $\alpha \in \mathbb{Z}_{\geq 0}^d$.  Then $\diffpower{I}{cn}\subseteq I^n$ where $c = \frac{(n-1)\alpha_\mathrm{max} + 1}{n}$ and $\alpha_{\mathrm{max}}=\max\{\alpha_i \mid i\in [d]\}$ for all $n$. 
\end{prop}
\begin{proof}
Let $S = \supp(\alpha)$.  By Proposition \ref{SingleMonomialDiffPower}, $\diffpower{I}{cn} = (\prod_{i\in S}x_i^{\alpha_i + cn - 1})$. 
We want to show that $\prod_{i\in S}x_i^{\alpha_i + cn - 1} \in I^n$. This is equivalent in showing that $\alpha_i + cn - 1 \ge \alpha_i n$ for all $i\in S$. Note that for all $i \in S$,
\begin{align*}
    \alpha_i + cn - 1 &= \alpha_i + \frac{(n-1)\alpha_\mathrm{max} + 1}{n}n - 1\\
    &= \alpha_i + (n-1)\alpha_\mathrm{max}\\
    &\ge \alpha_i n.
\end{align*}
Thus, $\diffpower{I}{cn}\subseteq I^n$.
\end{proof}
\begin{rmk}\label{SingleMonomialcontainment}
If $I = (x^{\alpha})$, then Proposition \ref{SingleMonomialContainmentDependsonN} also shows that $\diffpower{I}{\alpha_\mathrm{max}n}\subseteq I^n$ for all $n\in \mathbb{N}$.
\end{rmk}

\section{Differential Closure} \label{sec:diffClosure}

In this section, we introduce the differential closure of an ideal. The differential closure of an ideal is defined analogously to the integral closure or tight closure of an ideal, but using differential powers of an ideal instead of ordinary or Frobenius powers.

We start with some basic properties of differential closures and show they satisfy some of the axioms of a closure operation. Note that it's not clear whether differential closure is a closure operation in general, or whether the differential closure of an ideal is even another ideal. We show that both of these things are true when $R$ is a simple module over its ring of differential operators. This includes, for instance, strongly $F$-regular rings over a finite field.

\begin{defn}\label{closuredef}
Let $R$ be a commutative $A$-algebra domain and let $I \subseteq R$ be an ideal. We define the differential closure of $I$ (with respect to $A$) to be
\[
\diffidealA{I} = \{r \in R \mid \ \text{there exists} \ 0\neq c \in R : cr^n \in \diffpowerA{I}{n} \ \text{for all} \ n \gg 0 \}
\]
\end{defn}

It follows from this definition that the differential closure of an ideal contains its radical:

\begin{prop}[\cite{ZariskiNagataMixedChar}, Proposition 3.2 part 3]\label{inductivegeneral}\label{inductivegeneral2}
Let $R$ be a commutative $A$-algebra and let $I \subseteq R$ be an ideal, and let $n>0$. Then $\partial(I^n) \subseteq I^{n-t}$ for all $t$, $0 \leq t \leq n$, and $\partial \in D_A^t(R)$. In particular, taking $I$ to be a principal ideal $(r)$, we get $r^{n-t} \mid \partial(r^n)$.
\end{prop}
\begin{prop}\label{radicalcontainment}
Let $R$ be an $A$-algebra domain, and let $I \subseteq R$ be a ideal. Then $\sqrt{I} \subseteq \diffidealA{I}$.
\end{prop}
\begin{proof}

Let $r \in \sqrt{I}$. Then there exists a $k > 0$ such that $r^k \in I$. We want to show $r \in \diffidealA{I}$, or equivalently, there exists a nonzero $c$ such that for all $n$ sufficiently large and for all $\partial \in D_A^{n-1}(R)$, we have $\partial(cr^n) \in I$.

Let $n > 0$, let $\partial \in D^{n-1}_A(R)$, and let $c = r^k$.  Then by Proposition \ref{inductivegeneral2}, $r^{(k + n) - n} | \partial(cr^n)$, so that $r^k | \partial(cr^n)$.  In particular, $\partial(cr^n) \in I$ so that $r \in \diffidealA{I}$, as desired.
\end{proof}

Next, we ask whether differential closure is a closure operation. We will later show in Theorem \ref{idempotenceradical} that $\diffideal{I} = \sqrt{I}$ for any ideal $I$ when $R$ is a simple $D_A(R)$-module, and all of the properties below follow from this fact. However, we include a direct proof here for some properties which hold more generally.

\begin{prop}\label{isclosureopalmost}
If $R$ is any commutative $A$-algebra domain, then, 
\begin{enumerate}
    \item[(1)] $I \subseteq \diffidealA{I}$
    \item[(2)] if $I \subseteq J$, then $\diffidealA{I} \subseteq \diffidealA{J}$
    \item[(3)] For $x \in \diffidealA{I}$ and $r\in R$, we have $rx \in \diffidealA{I}$. Further, if $I \subseteq R$ has the property that $\diffpowerA{I}{m} \cdot \diffpowerA{I}{n} \subseteq \diffpowerA{I}{m+n}$ for all $m$ and $n$ (for instance, if $R$ is a polynomial ring over $A$) then $\diffidealA{I}$ is an ideal.
\end{enumerate}
\end{prop}

\begin{proof} $ $

\underline{Proof of (1):}
Suppose $r \in I$. We want to show there exists a nonzero $c$ such that $cr^n \in \diffpowerA{I}{n}$ for all $n$ sufficiently large.
Let $\partial \in D_A^{n-1}$, and 
choose $c = 1$. By Proposition \ref{inductivegeneral}, $r | \partial(r^n)$, so that $\partial(r^n) \in I$ which means $r^n \in \diffpowerA{I}{n}$ as desired.

\underline{Proof of (2):} 
Suppose $I \subseteq J$ and let $a\in \diffidealA{I}$. By definition there exists a nonzero $c\in R$ such that $ca^n\in \diffpowerA{I}{n}$ for all sufficiently large $n$. This implies that $\partial(ca^n) \in I \subset J$ for all $A$-linear differential operators $\partial$ with order less than $n$. This means that $ca^n\in \diffpowerA{J}{n}$ for all sufficiently large $n$. By definition $a \in \diffidealA{J}$. Thus, $\diffidealA{I} \subseteq \diffidealA{J}$.

\underline{Proof of (3):} 
First, we want to show that the differential closure is closed under multiplication by a ring element. If $a \in \diffidealA{I}$, want to show $ra \in \diffidealA{I}$. This means we want to find a nonzero $c \in R$ such that $c(ra)^n \in \diffpowerA{I}{n}$ for all $n$ sufficiently large. We have $(ra)^n = r^n a^n$. We know there exists a nonzero $c$ such that $ca^n \in \diffpowerA{I}{n}$. Since $r^n$ is a ring element and $\diffpowerA{I}{n}$ is an ideal, $r^n (ca^n) \in \diffpowerA{I}{n}$. 

Now, suppose that $I \subseteq R$ has the property that $\diffpowerA{I}{m} \cdot \diffpowerA{I}{n} \subseteq \diffpowerA{I}{m+n}$ for all $m$ and $n$. For $a,b\in \diffidealA{I}$, we want to show that $a+b\in\diffidealA{I}.$ We know from our definitions that there exist nonzero $c_a, c_b \in R$ and $N \in \mathbb{N}$ such that $c_a a^n \in \diffpowerA{I}{n}$ and $c_b b^{n} \in \diffpowerA{I}{n}$ for all $n \geq N$. We want to show $c_{(a+b)} (a+b)^{m}  \in \diffpowerA{I}{m}$ for some nonzero $c_{(a+b)} \in R$ and for all $m$ sufficiently large. 
We know by the binomial theorem that 
$$c_{(a + b)} (a + b)^{m} 
= c_{(a + b)} \sum_{k = 0}^{m} \binom{m}{k} a^k b^{m - k} 
= \sum_{k=0}^{m} \binom{m}{k} c_{(a + b)} a^k b^{m - k}$$

Since $c_a a^n \in \diffpowerA{I}{n}$ and $c_b b^n \in \diffpowerA{I}{n}$ for all $n\ge N$, we know that $(c_a a^N) \cdot a^i \in \diffpowerA{I}{i + N} \subseteq \diffpowerA{I}{i}$ and $(c_b b^N) \cdot b^i \in \diffpowerA{I}{i + N} \subseteq \diffpowerA{I}{i}$ for all $i$ greater than zero. In this case, let $c_{(a + b)} = c_a c_b a^N b^N$ then we have $c_{(a + b)} a^{k} b^{m - k} = c_a a^{k+N} c_b b^{m - k + N}$. We know $c_a a^{k + N} \in \diffpowerA{I}{k}$ from our choice of $N$. Similarly, $c_b b^{m - k  + N} \in \diffpowerA{I}{m - k}$. So, $c_{(a + b)} a^{k} b^{m - k} \in \diffpowerA{I}{k} \cdot \diffpowerA{I}{m - k} \subseteq \diffpowerA{I}{m}$ and we are done.
\end{proof}

\begin{rmk}
In the case of a polynomial ring over a field of characteristic 0, this shows that differential closure has all the properties of a closure operation other than idempotence. Idempotence is difficult to prove from the definitions even for polynomial rings. We don't know of a proof that differential closure is idempotent other than by showing it sometimes agrees with the radical, in Theorem \ref{radicalDmodule}.
\end{rmk}

The rest of this section is devoted to characterizing when the reverse containment of Proposition \ref{radicalcontainment} holds. We will need the following lemma describing how differential closure interacts with intersections. 

\begin{lemma}\label{diffclosureintersection}
Let $R$ be a commutative $A$-algebra and let $\{I_\alpha \mid \alpha \in S \}$ be a family of ideals in $R$. Then
  \[
\diffidealA{\bigcap_{\alpha\in S}I_{\alpha}} \subseteq
        \bigcap_{\alpha \in S} \diffidealA{I_{\alpha}}. 
  \]
When $S$ is finite, we have
  \[
        \diffidealA{\bigcap_{\alpha\in S}I_{\alpha}} =
        \bigcap_{\alpha \in S} \diffidealA{I_{\alpha}} .
  \]

\end{lemma}

\begin{proof}
First, we want to show that $\diffidealA{\bigcap_{\alpha\in S} I_\alpha}\subseteq\bigcap_{\alpha\in S} \diffidealA{I_\alpha}.$ We have $\bigcap_{\alpha\in S} I_\alpha\subseteq I_\alpha$ for every $\alpha\in S$. By Proposition \ref{isclosureopalmost}, this implies that $\diffidealA{\bigcap_{\alpha\in S}I_\alpha} \subseteq \diffidealA{I_\alpha}$ for every $\alpha\in S$. Thus, by the definition of intersection, we have
\[
\diffidealA{\bigcap_{\alpha\in S} I_\alpha}\subseteq\bigcap_{\alpha\in S} \diffidealA{I_\alpha}.
\]
Now we want to show that $\diffidealA{\bigcap_{\alpha\in S} I_\alpha}\supseteq\bigcap_{\alpha\in S} \diffidealA{I_\alpha}$ when $S$ is finite. Let $b\in \bigcap_{\alpha\in S} \diffidealA{I_\alpha}$ be an arbitrary element. This means that $b\in \diffidealA{I_{\alpha}}$ for all $\alpha\in S$. By definition, for all $\alpha\in S$ there exists an element $c_\alpha \in R$ such that $c_\alpha b^n \in \diffpowerA{I_{\alpha}}{n}$ for all $n$ sufficiently large. This implies that $(\prod_{\alpha\in S} c_{\alpha}) b^n \in \diffpowerA{I_\alpha}{n}$ for all $\alpha\in S$ and $n$ sufficiently large. 
By definition of intersection and Lemma \ref{exercise2.13}, we have 
\[
  \left(\prod_{\alpha\in S} c_{\alpha}\right) b^n 
  \in \bigcap_{\alpha\in S} \diffpowerA{I_\alpha}{n} 
  = \diffpowerA{\left(\bigcap_{\alpha\in S} I_{\alpha}\right)}{n}
\]
for all $n$ sufficiently large.  By the definition of differential closure, $b\in \diffidealA{\bigcap_{\alpha\in S} I_{\alpha}}$. This implies that $\diffidealA{\bigcap_{\alpha\in S} I_\alpha}\supseteq\bigcap_{\alpha\in S} \diffidealA{I_\alpha}.$
Thus, $\bigcap_{\alpha \in S} \diffidealA{I_{\alpha}} = \diffidealA{\bigcap_{\alpha\in S}I_{\alpha}}.$
\end{proof}

\begin{lemma}[{\cite[Proposition 2.17, Lemma 3.10]{BRENNER2019106843} }] \label{lemma:diffPowerLocalization}
Let $A$ be a Noetherian ring and let $R$ be an $A$ algebra essentially of finite type. Let $W\subseteq R$ be a multiplicative set. Then $W\invrs D_A(R) = D_A(W\invrs R)$. Further, we have $W\invrs \diffpowerA{I}{n} = (W\invrs I)^{\langle n \rangle_{W\invrs R/A}}$ for all $n$.
\end{lemma}

\begin{lemma} \label{lemma:diffClosureLocalization}
Let $A$ be a Noetherian ring and let $R$ be a domain essentially of finite type over $A$. Let $W\subseteq R$ be a multiplicative set. Then we have an inclusion of sets:
\[
   \left\{ w\invrs x \in W\invrs R \mid x\in \diffidealA{I} \right\} \subseteq \diffidealA{W\invrs I}
\]
\end{lemma}
\begin{proof}
  Let $x \in \diffidealA{I}$. By assumption, there exists a nonzero $c\in R$ with $cx^n \in \diffpowerA{I}{n}$ for all $n \gg 0$. By Lemma \ref{lemma:diffPowerLocalization}, this implies that $cx^n \in (W\invrs I)^{\langle n \rangle_{W\invrs R/A}}$, and so $x \in \diffidealA{W\invrs I}$. But then we're done, by Proposition \ref{isclosureopalmost}.
\end{proof}

Note that $R$ is a naturally a module over the (non-commutative) ring $D_A(R)$. We say that $R$ is a \emph{simple} $D_A(R)$-module if $R$ has no $D_A(R)$-submodules other than $0$ and $R$ itself. Note that any $D_A(R)$-submodule of $R$ must be an ideal of $R$, as multiplication by an element of $R$ is a differential operator. Further, $R$ is a simple $D_A(R)$-module if and only if, for all nonzero $c\in R$, there exists some $\partial \in D_A(R)$ with $\partial(c) = 1$. One can show, for instance by using Lemma \ref{lemma:diffPowerLocalization}, that $W\invrs R$ is a simple $D_A(W\invrs R)$-module when $R$ is a simple $D_A(R)$-module.

\begin{prop}\label{maximalclosure}
Let $A$ be a Noetherian ring and $R$ a domain essentially of finite type over $A$, and suppose $R$ is a simple $D_A(R)$-module. Let $\mathfrak{p} \subseteq R$ be a prime ideal. Then $\mathfrak{p} = \diffidealA{\mathfrak{p}}.$
\end{prop}
\begin{proof}
From Proposition \ref{isclosureopalmost}, we know that $\mathfrak{p} \subseteq \diffideal{\mathfrak{p}}$, so it suffices to prove $\diffideal{\mathfrak{p}} \subseteq \mathfrak{p}$. So suppose that $\diffideal{\mathfrak {p}}$ contains an element not in $\mathfrak p$. Then Lemma \ref{lemma:diffClosureLocalization} tells us that $\diffidealA{\mathfrak pR_{\mathfrak p}}$ contains a unit $u$, and by Proposition \ref{isclosureopalmost} we have $1 \in \diffidealA{\mathfrak pR_{\mathfrak p}}$. Thus, there exists a nonzero $c$ such that for all $n$ sufficiently large, $c\in \diffpowerA{(\mathfrak pR_{\mathfrak p})}{n}$. This means $\partial(c) \in \mathfrak pR_{\mathfrak p}$ for all $\partial\in D_A(R_{\mathfrak p})$. But $R_{\mathfrak p}$ is a simple $D_A(R_{\mathfrak p})$-module and $c\neq 0$, so we must have $\partial(c) = 1$ for some  $\partial\in D_A(R_{\mathfrak p})$, a contradiction. 
\end{proof}

\begin{thm}\label{idempotenceradical}
Let $A$ be a Noetherian ring and let $R$ be an $A$ algebra essentially of finite type. Suppose $R$ is a simple $D_A(R)$-module, and $I \subseteq R$ is an ideal. Then 
\[
  \diffidealA{I} = \sqrt{I}
\]
\end{thm}

\begin{proof}
We already know that $\sqrt{I} \subseteq \diffidealA{I}$ from Proposition \ref{radicalcontainment}, so it suffices to show that $\diffidealA{I}\subseteq \sqrt I$. 
From Proposition \ref{isclosureopalmost}, we have 
\[
  \diffidealA{I} \subseteq \diffidealA{\sqrt I}.
\]
Because $R$ is Noetherian, we know
\[
 \sqrt I = \bigcap_{I \subseteq \mathfrak p} \mathfrak p,
\]
where $\mathfrak p$ ranges over the prime ideals containing $I$. Thus 
\[
  \diffidealA{\sqrt I} = \diffidealA{\bigcap_{I \subseteq \mathfrak p} \mathfrak p} \subseteq \bigcap_{I \subseteq \mathfrak p} \diffidealA{\mathfrak p},
\]
by Proposition \ref{diffclosureintersection}. Finally, we can use Proposition \ref{maximalclosure} to see that
\[
  \bigcap_{I \subseteq \mathfrak p} \diffidealA{\mathfrak p} = \bigcap_{I \subseteq \mathfrak p} \mathfrak p = \sqrt I,
\]
as desired.
\end{proof}

\begin{cor}\label{radicalDmodule}
Let $A$ be a Noetherian ring and let $R$ be an $A$ algebra essentially of finite type. Suppose that $0\neq I \subseteq R$ is a $D_A(R)$-submodule of $R$. Then $\diffidealA{I} = R$. In particular, we have $\sqrt{J} = \diffideal{J}$ for all ideals $J \subseteq R$ if and only if $R$ is a simple $D_A(R)$-module

\begin{proof}
Let $c\in I$ be any nonzero element and let $r \in R$. Then $cr^n \in I$ for all and $n>0$, because $I$ is an ideal. As $I$ is a $D_A(R)$-submodule of $R$, we know that $\partial(I) \subseteq I$ for all $\partial \in D_A(R)$. In particular, we have $\partial(cr^n) \in I$ for all differential operators $\partial$ of order less than $n$. This shows that $r \in \diffidealA{I}$, as desired. 
\end{proof}
\end{cor}

\bibliographystyle{alpha}
\bibliography{bibliography}

\end{document}